\documentclass[12pt,a4paper]{article}
\usepackage{}
\usepackage{mathrsfs}
\usepackage{amscd,amssymb,hyperref}
\usepackage{amsmath}
\usepackage{amsthm}
\usepackage{graphics}
\usepackage{epsfig}
\usepackage{picinpar}
\usepackage{wrapfig}
\usepackage{amsmath}
\usepackage{color}
\usepackage{pstricks,pst-node,pst-text,pst-3d}
\usepackage{pst-plot}
\usepackage{verbatim}
\usepackage{graphicx}
\newtheorem{theorem}{Theorem}[section]
\newtheorem{thm}[theorem]{Theorem}
\newtheorem{lem}[theorem]{Lemma}

\newtheorem{prop}[theorem]{Proposition}
\newtheorem{dfn}[theorem]{Definition}
\newtheorem{rem}[theorem]{Remark}
\newtheorem{exa}[theorem]{Example}

\numberwithin{equation}{section}
\usepackage[textheight=22cm, textwidth=14.5cm]{geometry}%

\title{On the annulus complex of a handlebody}

\author{Dongqi Sun\thanks{Supported by a grant of NSFC (Grant No.12101153)}\\College of Mathematical Sciences, Harbin Engineering University,\\ Harbin 150001, CHINA. \\Email: sundq1029@hrbeu.edu.cn\\
}

\date{}

\begin{document}

\maketitle

\begin{abstract}
In this paper, we give the definition of the annulus complex of a handlebody and use the combinatorial methods to prove its connectivity.
\end{abstract}

\baselineskip 18.5pt
\section{Introduction}

For an orientable closed surface, Harvey\cite{Harvey} introduced its curve complex. Since then, many other simplicial complexes appears. A common characteristic of these complexes is their connectivity.

Similarly, for a handlebody, some complexes are also defined.
Let $H_g$ be an orientable handlebody of genus $g$.
It is known that the disk complex $\mathcal{D}$ of $H_g$ is connected, namely, any two essential disks in $\mathcal{D}$ are connected by a path in $\mathcal{D}$. In fact, $\mathcal{D}$ is contractible, see \cite{McCullough}. In addition, other complexes of $H_g$ have been demonstrated to be connected, such as cut system complex, pants complex, separating disk complex, half disk complex, disk pants graph, nonseparating disk pants graph and so on. For example, see \cite{Cho-Koda,Johnson,Wajnryb,Sun} .

There are some useful tools to prove the connectivity of these complexes, such as curve surgery, Morse/Cerf theroy, and Teichm$\ddot{u}$ller theory. For example, see
\cite{Farb,Harer,H-T,Hatcher,Iva,M-S,Penner}.
Putman\cite{Put} proved the connectivity using the basic combinatorial group theory, which was also adopted in\cite{Guo-Liu}. Putman's method requires the use of a set of appropriate generators of the mapping class group of $\partial H_g$.

In the present paper, we define the annulus complex of a handlebody $H_g$, and
then prove the connectivity of the annulus complex using combinatorial topology methods. As defined below, the vertices of the annulus complex are essential annuli in $H_g$, and the boundary curves of these annuli are complicated.
Consequently, employing other methods to establish the connectivity of the annulus complex is challenging. So this article uses the traditional combinatorial methods.

This article is organized as follows. In Section 2, we introduce some basic definitions and some conclusions. In Section 3, we discuss the properties of essential annuli in a handlebody. In Section 4, we prove the connectivity of the annulus complex of a handlebody, see Theorem\ref{thm3}.

Throughout this paper, the number of elements in a set $C$ is denoted by $|C|$.
The intersection number of the objects $A$ and $B$ is denoted by $|A \cap B|$. We always assume $|A\cap B|$ is minimal, namely,
$|A \cap B|=\min\{|A'\cap B'|~|A'~is~isotopic~to~A,\\~B'~is~isotopic~to~B\}$.
The interior of the object $A$ is denoted by $int(A)$.

The definitions and terminologies not defined here are standard. See, for example, \cite{Hempel} or \cite{Jaco}.

\section{Preliminaries}

In this section, we introduce some basic definitions and useful results.

\begin{dfn}

Let $H_g$ be an orientable handlebody of genus $g\geq 2$, and $D$ be an essential disk in $H_g$. Let $\alpha$ be a simple arc in $\partial H_g$ with $\alpha\cap D=\alpha\cap \partial D=\partial \alpha$ and both points of $\partial \alpha$ meeting $\partial D$ from the same side of $\partial D$. Let $N(\alpha)=\alpha\times [0,1]$ be the regular neighborhood of $\alpha$ in $\partial H_g$, $\alpha=\alpha\times \{\frac{1}{2}\}$, and $N(\alpha)\cap D=\partial \alpha\times [0,1]$.
Connect $N(\alpha)$ and $D$ along $\partial \alpha\times [0,1]$ and push $\alpha\times (0,1)$ into $int(H_g)$ slightly. Then we get a properly embedded annulus denoted by $A=D\cup (\alpha\times [0,1])$.
We say that $A$ is the band-sum of $D$ along $\alpha$. This process is called doing a band-sum to disk $D$ along an arc $\alpha$ to get an annulus. See Fig.\ref{fig7}.

\end{dfn}

\begin{figure}[htbp]
\centering
\includegraphics[scale=0.8]{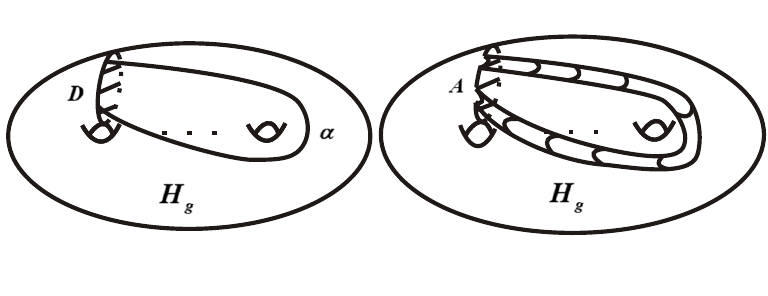}
\caption{Do band-sum to $D$ along $\alpha$ in $H_g$.}\label{fig7}
\end{figure}

From the above definition, we know that if an essential annulus $A=D\cup (\alpha\times [0,1])$, then $\partial-$compressing $A$ along a disk which intersects $\alpha$ transversely nonempty can regain the essential disk $D$.

\begin{dfn}

Let $H_g$ be an orientable handlebody of genus $g\geq 2$, and $D$ be an essential disk in $H_g$. Let $\alpha$ be a simple arc in $\partial H_g$ with $\alpha\cap D=\alpha\cap \partial D=\partial \alpha$ and both points of $\partial \alpha$ meeting $\partial D$ from the same side of $\partial D$.
If doing band-sum to $D$ along $\alpha$ results in an essential annulus in $H_g$, then the arc $\alpha$ is called an essential arc associated with $D$.

\end{dfn}

Essential annulus is a type of simple essential orientable surface in the handlebody. Let $\mathfrak{A}$ be a collection of pairwise disjoint non-parallel essential annuli in $H_g$. $\mathfrak{A}$ is {\it maximal} if whenever $A$ is an essential annulus in $H_g$ with $A\cap \mathfrak{A}=\phi$ then $A$ is parallel to a component of $\mathfrak{A}$ in $H_g$.
Rubinstein-Scharlemann\cite{R-S},
Lei-Tang\cite{Lei-Tang} and Yin-Tang-Lei\cite{Yin-Tang-Lei} studied the $|\mathfrak{A}|$.

\begin{prop}\cite{R-S}\label{prop2.5}
Let $H_2$ be an orientable handlebody of genus $2$. Then a maximal collection of essential annuli $\mathfrak{A}$ could contain exactly $1$, or $2$, or at most $3$ annuli.
\end{prop}

\begin{prop}\cite{Lei-Tang,Yin-Tang-Lei}
Let $H_g$ be an orientable handlebody of genus $g\geq 3$. Then for a maximal collection of essential annuli $\mathfrak{A}$, $2\leq|\mathfrak{A}|\leq 4g-5$ and the bound is best.
\end{prop}

Now the annulus complex of a handlebody is defined as follows.

\begin{dfn}

Let $H_g$ be an orientable handlebody of genus $g\geq2$. The annulus complex $\mathcal{A}$ of $H_g$ is a simplicial complex whose vertices are the isotopy classes of essential annuli in $H_g$ and
a collection of $k+1$ distinct vertices constitute a $k-$simplex if there are pairwise disjoint representatives.

\end{dfn}

Similar to the dimension of the disk complex, the {\it dimension} of the annulus complex $\mathcal{A}$ is defined to be the biggest dimension of the simplices in $\mathcal{A}$.

By Propositon\ref{prop2.5}, we have the following proposition about the dimension of the annulus complex $\mathcal{A}$.

\begin{prop}

Let $H_g$ be an orientable handlebody of genus $g\geq2$. The dimension of the annulus complex $\mathcal{A}$ is $4g-6$.

\end{prop}

\section{The properties of essential annuli in a handlebody}

In this section, several properties of the essential annuli in $H_g$ are considered.

\begin{prop}
Let $H_g$ be an orientable handlebody of genus $g\geq2$, and $D$ be an essential disk in $H_g$. Then there are at least two different essential
annuli $A_1$ and $A_2$ with $A_1\cap A_2=\phi$, both of which are obtained by doing band-sum to $D$ along disjoint arcs.
\end{prop}

\begin{proof}
There are two cases for the essential disk $D$ to be considered.

\textbf{Case} 1: $D$ is separating in $H_g$.

Assume $H_g\setminus D=H_{g_1}\cup H_{g_2}$ with $g(H_{g_i})=g_i\geq 1 (i=1,2)$ and $g_1+g_2=g$. Then there exist essential arcs $\alpha_i (i=1,2)$
associated with $D$ such that $\alpha_1\subset \partial H_{g_1}$ and $\alpha_2\subset \partial H_{g_2}$. See Fig.\ref{fig1}.

\begin{figure}[htbp]
\centering
\includegraphics[scale=0.6]{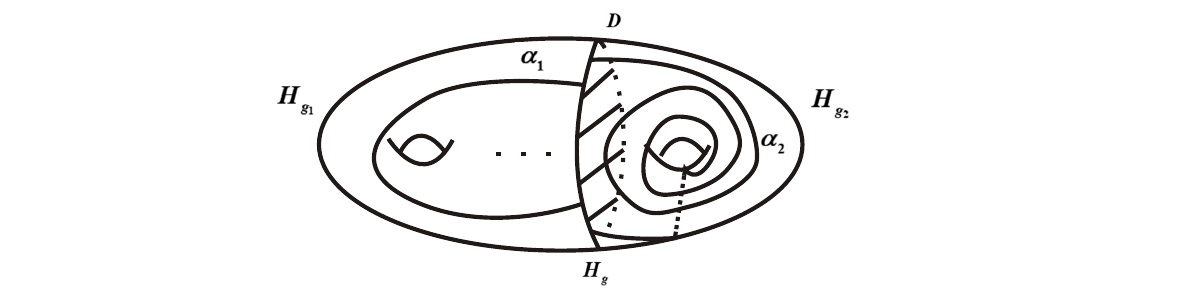}
\caption{$D$ is separating in $H_g$.}\label{fig1}
\end{figure}

Let $A_1=D\cup (\alpha_1\times [0,1])$ and $A_2=D\cup (\alpha_2\times [0,1])$.
It is
easy to see that $A_1$ is not isotopic to $A_2$ and $A_1\cap A_2=\phi$.

\textbf{Case 2}: $D$ is nonseparating in $H_g$.

Now $H_g\setminus D=H'$ is a handlebody of genus $g(H')=g-1$. Assume the two cutting sections of $D$ on $\partial H'$ are denoted
by $D'$ and $D''$. Since $g\geq 2$, $\chi(\partial H'\setminus(D_1\cup D_2))=2-2(g-1)-2=2-2g\leq-2$. Thus, there exist two disjoint essential arcs $\alpha_1$ and $\alpha_2$ associated with $D$
such that $\partial \alpha_1 \in \partial D'$ and $\partial \alpha_2 \in \partial D''$. See Fig.\ref{fig2}.

\begin{figure}[htbp]
\centering
\includegraphics[scale=0.6]{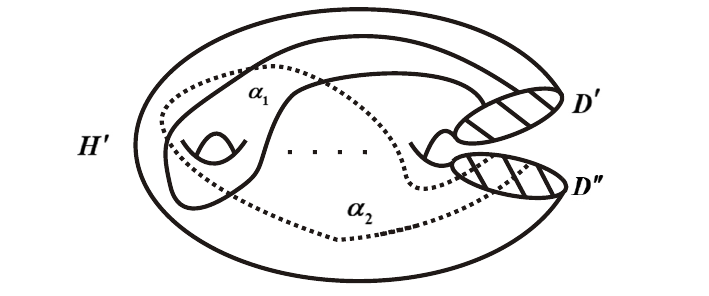}
\caption{$D$ is nonseparating in $H_g$.}\label{fig2}
\end{figure}

Now $A_1=D\cup (\alpha_1\times [0,1])$ and $A_2=D\cup (\alpha_2\times [0,1])$ are two essential annuli in $H_g$. It is easy to see that $A_1$ is not isotopic to $A_2$ and $A_1\cap A_2=\phi$.

\end{proof}

\begin{prop}

Let $H_g$ be an orientable handlebody of genus $g\geq 2$, and $A$ be an essential annulus in $H_g$. If $A=D\cup (\alpha\times [0,1])$, then for $A$ and $D$, either both are separating in $H_g$ or both are nonseparating in $H_g$.

\end{prop}

\begin{proof}

There are two cases for $D$ to be considered.

\textbf{Case 1}: $D$ is separating in $H_g$.

Let $H_g\setminus D=H_{g_1}\cup H_{g_2}$ with $g_i\geq 1 (i=1,2)$ and $g_1+g_2=g$. Without loss of generality, assume
$\alpha \subset \partial H_{g_2}$. Then $H_g\setminus A$ contains two components with one of which isotopic to $H_{g_2}$ and the other isotopic to the handlebody obtained by adding a 1-handle to $H_{g_1}$. So $A$ is separating in $H_g$.

\textbf{Case 2}: $D$ is nonseparating in $H_g$.

Let $H_g\setminus D=H'$. So $H'$ is a handlebody of genus $g(H')=g-1$. Assume the two cutting sections of $D$ on $\partial H'$ are denoted by
$D'$ and $D''$. Without loss of generality, assume $\partial \alpha\in \partial D'$. Then $H_g\setminus A$ is isotopic to the handlebody which is obtained by adding a 1-handle to $H'$. So $A$ is nonseparating in $H_g$.

\end{proof}

\begin{prop}\label{prop3.3}

Let $A$ be an essential annulus in $H_g$. Then $\partial-$compressing $A$ may result in different essential disks in $H_g$.

\end{prop}

The following is an example of Proposition\ref{prop3.3}.

\begin{exa}

As in Fig.\ref{fig13}, $D_1$ and $D_2$ are two different nonseparating essential disks in $H_5$. $\alpha_1$ and $\alpha_2$ are all simple arcs connecting $\partial D_1$ and $\partial D_2$ on $\partial H_5$.
Push $int(D_1\cup (\alpha_1\times [0,1])\cup D_2)$ into $int(H_5)$ to obtain an essential disk $D'$ in $H_5$.
Similarly, push $int(D_1\cup (\alpha_2\times [0,1])\cup D_2)$ into $int(H_5)$ to obtain an essential disk $D''$ in $H_5$. Then $D'$ and $D''$ are disjoint different essential disks in $H_5$.
Now doing band-sum to $D'$ along $\alpha_2$ can result in an essential annulus $A$ in $H_5$. It is easy to see that $A$ can also be obtained by doing band-sum to $D''$ along $\alpha_1$. In other words, $\partial-$compressing $A$ can result in $D'$ and $D''$, and these two disks are different.

\end{exa}

\begin{figure}[htbp]
\centering
\includegraphics[scale=0.6]{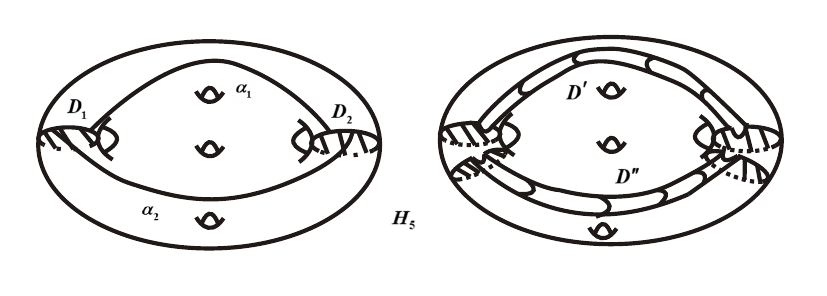}
\caption{$\partial-$compressing $A$ results in different disks.}\label{fig13}
\end{figure}

\section{The connectivity of the annulus complex}

Firstly, consider the annulus complex of the orientable handlebody of genus two.

\begin{thm}\label{thm2.7}
Let $H_2$ be an orientable handlebody of genus $2$. Then the annulus complex $\mathcal{A}$ of $H_2$ is not connected.
\end{thm}

\begin{proof}
By Proposition\ref{prop2.5}, there exists a maximal collection of essential annuli, say $\mathfrak{A}$, in $H_2$ which contains only one essential annulus. Assume $\mathfrak{A}=\{A\}$. Since $\mathfrak{A}$ is maximal, for any other essential annulus $A'$ in $H_2$ which is different from $A$, we know that $A\cap A'\neq \phi$. Thus, for any $B\in \mathcal{A}$, $A$ and $B$ can't be connected by a path in $\mathcal{A}$. So $\mathcal{A}$ is not connected.
\end{proof}

\begin{rem}
For the solid torus, there doesn't exist essential annulus in it. So the annulus complex of the solid torus is empty. By Theorem\ref{thm2.7}, we know the annulus complex $\mathcal{A}$ of $H_2$ is not connected. So our main theorem is about the orientable handlebody of genus at least $3$.
\end{rem}

From now on, we always assume the genus of $H_g$ is at least $3$. All the arcs considered here are simple.

To detect the connectivity of $\mathcal{A}$, we need to discuss the intersections of annuli in $\mathcal{A}$.
Since $H_g$ is an irreducible 3-manifold, using the usually cutting-pasting methods in combinatorial topology, we can assume each component of $A_1\cap A_2$ is an arc in $A_i (i=1,2)$ in the sense of isotopy.

We first prove the connectivity of $\mathcal{A}$ for two special cases, and then prove the connectivity of $\mathcal{A}$ for the general case.

\begin{lem}\label{lem1}

If each component of $A_1\cap A_2$ is an arc with boundary points belonging to the same component of $\partial A_i$ for $i=1,2$, then $A_1$ and $A_2$ are connected in $\mathcal{A}$.

\end{lem}

\begin{proof}

Assume $|A_1 \cap A_2|=m$. The conclusion can be obtained by the induction on $m$.

If $m=0$, then $A_1$ and $A_2$ are disjoint. Thus, $A_1$ and $A_2$ are connected in $\mathcal{A}$.

Assume for $m=k>0$, the conclusion is correct.

When $m=k+1$. Choose one component of $A_1\cap A_2$, say $\alpha$, which is outermost relative to $A_2$. Namely, one component of $A_2\setminus \alpha$ is a disk, say $E_1$, which intersects $A_1$ only in the arc $\alpha$. Now $A_1\setminus \alpha$ contains two components with one of which a disk, say $E_2$.

Now let
$A_1'=(A_1\setminus E_2) \cup E_1$. In the sense of isotopy, $A_1'$ is an essential annulus in $H_g$ with $A_1'\cap A_1=\phi$ and
$|A_1'\cap A_2|< |A_1\cap A_2|$. So $|A_1'\cap A_2|\leq k$.
Then by the induction, there exists a path $A_1'-\cdots -A_2$ in $\mathcal{A}$ from $A_1'$ to $A_2$.
So there exists a path $A_1-A_1'-\cdots -A_2$ in $\mathcal{A}$ from $A_1$ to $A_2$.

\end{proof}

Because of Lemma\ref{lem1}, for any essential annuli $A_1,A_2\in \mathcal{A}$  with $A_i=D_i\cup (\alpha_i\times [0,1]) (i=1,2)$, we can assume $(\alpha_1\times ([0,\frac{1}{2})\cup (\frac{1}{2},1]))\cap (\alpha_2\times ([0,\frac{1}{2})\cup (\frac{1}{2},1]))=\phi$.

\begin{thm}\label{thm1}

Let $H_g$ be an orientable handlebody of genus $g\geq 3$, and $D$ be an essential disk in $H_g$. Suppose $A_1=D\cup (\alpha_1\times [0,1])$ and $A_2=D\cup (\alpha_2\times [0,1])$. Then $A_1$ and $A_2$ are connected by a
path in the annulus complex $\mathcal{A}$ of $H_g$.
\end{thm}

\begin{proof}

There are two cases to be considered.

\textbf{Case 1:} $\alpha_1$ and $\alpha_2$ are on the same side of $D$.

\textbf{Subcase 1.1:} $D$ is separating in $H_g$. Let $H_g\setminus D=H_{g_1}\cup H_{g_2}$ with $g_i\geq1(i=1,2)$ and $g_1+g_2=g$. See Fig.\ref{fig8}(a).

\begin{figure}[htbp]
\centering
\includegraphics[scale=0.6]{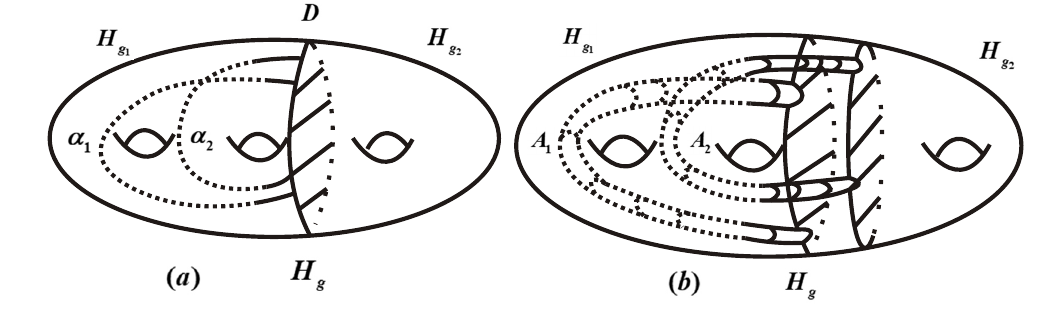}
\caption{$D$ is separating with $\alpha_1$ and $\alpha_2$ on the same side of $D$.}\label{fig8}
\end{figure}

Without loss of generality, suppose $\alpha_1$ and $\alpha_2$ are all contained in $\partial H_{g_1}$.
As in Fig.\ref{fig8}(b), $A_1\cap A_2\neq \phi$.
Then we can choose an essential arc $\alpha_3\subset\partial H_{g_2}$ associated with $D$ to make the annulus $A_3=D\cup (\alpha_3\times [0,1])$ essential in $H_g$. Now $A_3\cap A_1=A_3\cap A_2=\phi$. So $A_1-A_3-A_2$ is a path in $\mathcal{A}$ from $A_1$ to $A_2$.

\textbf{Subcase 1.2:} $D$ is nonseparating in $H_g$. Let $H_g\setminus D=H'$. Then $H'$ is a genus $g-1$ handlebody. Assume the two cutting sections of $D$ on $\partial H'$ are denoted by $D^{'}$ and $D^{''}$. Without loss of generality, assume $\partial \alpha_i\in \partial D^{'}$ for $i=1,2$. See Fig.\ref{fig9}.

\begin{figure}[htbp]
\centering
\includegraphics[scale=0.6]{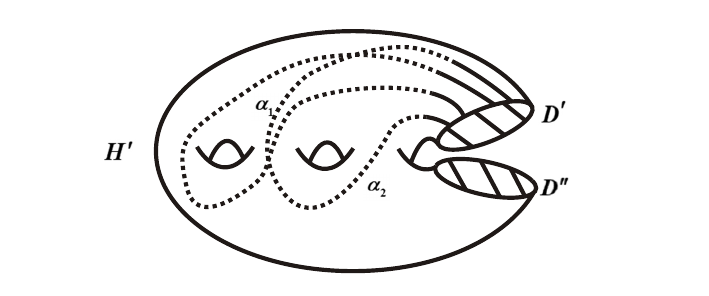}
\caption{$D$ is nonseparating with $\alpha_1$ and $\alpha_2$ on the same side of $D$.}\label{fig9}
\end{figure}

Now consider $|\alpha_1\cap \alpha_2|$. The conclusion is obtained by the induction on $|\alpha_1\cap \alpha_2|$.

When $|\alpha_1\cap \alpha_2|=0$. It is easy to see that $|A_1\cap A_2|=2$. See Fig.\ref{fig10}.

\begin{figure}[htbp]
\centering
\includegraphics[scale=0.6]{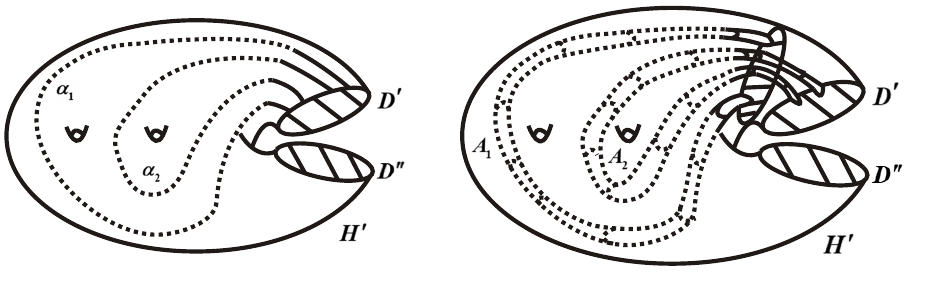}
\caption{$D$ is nonseparating with $\alpha_1$ and $\alpha_2$ on the same side of $D$ and $\alpha_1\cap\alpha_2=\phi$.}\label{fig10}
\end{figure}

At this time, $\chi ((\partial H'\setminus(D'\cup D''))\setminus (\alpha_1\cup \alpha_2))=2-2(g-1)-2+2=4-2g\leq -2$. So we can choose an essential properly embedded arc, say $\alpha_3$, associated with $D$ in $\partial H'\setminus(D'\cup D'')$ such that $\partial \alpha_3 \in \partial D^{''}$ and $\alpha_1\cap \alpha_3=\alpha_2\cap \alpha_3=\phi$. Now
$A_3=D\cup (\alpha_3\times [0,1])$ is an essential annulus in $H_g$ with $A_1\cap A_3=A_2\cap A_3=\phi$. So $A_1-A_3-A_2$ is a path in $\mathcal{A}$ from $A_1$ to $A_2$. The conclusion is correct.

Assume for $|\alpha_1\cap \alpha_2|=k>0$, the conclusion is correct.

When $|\alpha_1\cap \alpha_2|=k+1$. Let $\partial \alpha_1=\{P_1, P_2\}$ and $\partial \alpha_2=\{P_3, P_4\}$.
Without loss of generality, choose the point $P_3$.
Now there exists a point of $\alpha_1\cap \alpha_2$, say $P$, which is closest to $P_3$ in $\alpha_2$. Here, $P$ closest to $P_3$ in $\alpha_2$ means that $\alpha_2\setminus P=\gamma\cup \gamma'$, $\partial \gamma=\{P, P_3\}$ and $int(\gamma)\cap \alpha_1=\phi$.
Let $\alpha_1\setminus P=\beta_1\cup\beta_2$. Now there are two essential properly embedded arcs $(\alpha_1\setminus \beta_1) \cup \gamma=\beta_2\cup \gamma$ and
$(\alpha_1\setminus \beta_2) \cup \gamma=\beta_1\cup \gamma$ associated with $D$. Push these two arcs slightly to make them disjoint from $\alpha_1$.
Let $\alpha_3$ denote any one of these two arcs. It is easy to see that
$| \alpha_3\cap \alpha_2|< |\alpha_1\cap \alpha_2|$. So $|\alpha_3\cap \alpha_2|\leq k$.
Let $A_3=D\cup (\alpha_3\times [0,1])$.
By the induction,
there exists a path $A_3-\cdots-A_2$ in $\mathcal{A}$ from $A_3$ to $A_2$.
Since $|\alpha_1\cap \alpha_3|=0$, similar to the above discussion, there exists a path $A_1-\cdots-A_3$ in $\mathcal{A}$ from $A_1$ to $A_3$.
Thus, there exists a path $A_1-\cdots-A_3-\cdots-A_2$ in $\mathcal{A}$ from $A_1$ to $A_2$.

\textbf{Case 2:} $\alpha_1$ and $\alpha_2$ are on different sides of $D$.

\textbf{Subcase 2.1:} $D$ is separating in $H_g$. Let $H_g\setminus D=H_{g_1}\cup H_{g_2}$ with $g_i\geq1 (i=1,2)$ and $g_1+g_2=g$. See Fig.\ref{fig1}.

Without loss of generality, assume $\alpha_1\subset \partial H_{g_1}$ and $\alpha_2\subset \partial H_{g_2}$. Then $A_1\cap A_2=\phi$. So $A_1$ and $A_2$ are connected in $\mathcal{A}$.

\textbf{Subcase 2.2:} $D$ is nonseparating in $H_g$. Let $H_g\setminus D=H'$.
Then $H'$ is a handlebody of genus $g-1\geq 2$. Assume the two cutting sections of $D$ on $\partial H'$ are denoted by $D'$ and $D''$. Without loss of generality, assume $\partial\alpha_1\in \partial D'$ and $\partial\alpha_2\in \partial D''$. See Fig.\ref{fig11}.

\begin{figure}[htbp]
\centering
\includegraphics[scale=0.6]{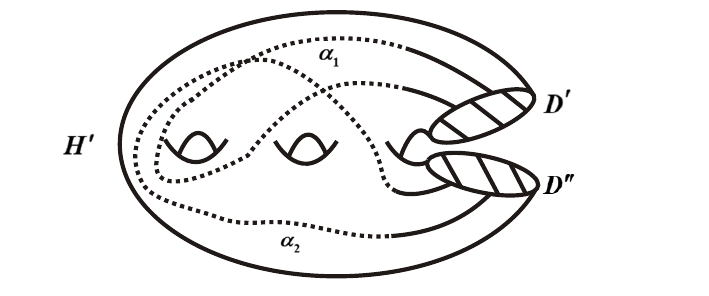}
\caption{$D$ is nonseparating with $\alpha_1$ and $\alpha_2$ on different sides of $D$.}\label{fig11}
\end{figure}

If $\alpha_1 \cap \alpha_2=\phi$, then $A_1\cap A_2=\phi$ and $A_1$ and $A_2$ are connected in $\mathcal{A}$.
If $\alpha_1 \cap \alpha_2\neq \phi$, then $A_1\cap A_2\neq \phi$. At this time, $\chi (\partial H'\setminus(D'\cup D''))=2-2(g-1)-2=2-2g\leq-4$.
So we can choose two disjoint essential properly embedded arcs, say $\alpha_1^{'}$ and $\alpha_2^{'}$, associated with $D'$ and $D''$ respectively in $\partial H'$ such that $\partial \alpha_1^{'}\in \partial D'$ and $\partial \alpha_2^{'}\in \partial D''$.
So $A_1 ^{'}=D\cup (\alpha_1^{'}\times [0,1])$ and
$A_2 ^{'}=D\cup (\alpha_2^{'}\times [0,1])$ are disjoint essential annuli in $H_g$, and $A_1 ^{'}$ and $A_2 ^{'}$ are connected in $\mathcal{A}$.
Now $A_1$ and $A_1 ^{'}$ are all obtained by doing band-sum to $D$ along arcs from the same side of $D$. By Case 1, there exists a path $A_1-\cdots-A_1 ^{'}$ in $\mathcal{A}$ from $A_1$ to $A_1 ^{'}$.
Similarly, $A_2 ^{'}$ and $A_2$ are all obtained by doing band-sum to $D$ along arcs from the same side of $D$. By Case 1, there exists a path $A_2 ^{'}-\cdots-A_2$ in $\mathcal{A}$ from $A_2 ^{'}$ to $A_2$.
So there exists a path $A_1-\cdots-A_1 ^{'}-A_2 ^{'}-\cdots-A_2$ in $\mathcal{A}$ from $A_1$ to $A_2$.

\end{proof}

\begin{thm}\label{thm2}
Let $H_g$ be an orientable handlebody of genus $g\geq 3$, and $D_1$ and $D_2$ be different essential disks in $H_g$ with $D_1\cap D_2=\phi$. Suppose $A_1=D_1\cup (\alpha_1\times [0,1])$ and $A_2=D_2\cup (\alpha_2\times [0,1])$. Then $A_1$ and $A_2$ are connected by a path in the annulus complex $\mathcal{A}$ of $H_g$.
\end{thm}

\begin{proof}

There are three cases to be considered.

\textbf{Case 1:} Both $D_1$ and $D_2$ are separating in $H_g$.

Let $H_g\setminus (D_1\cup D_2)=H_{g_1}\cup H_{g_2}\cup H_{g_3}$ with $g_1+g_2+g_3=g$ and $g_i\geq 1 (i=1,2,3)$. See Fig.\ref{fig4}.

\begin{figure}[htbp]
\centering
\includegraphics[scale=0.6]{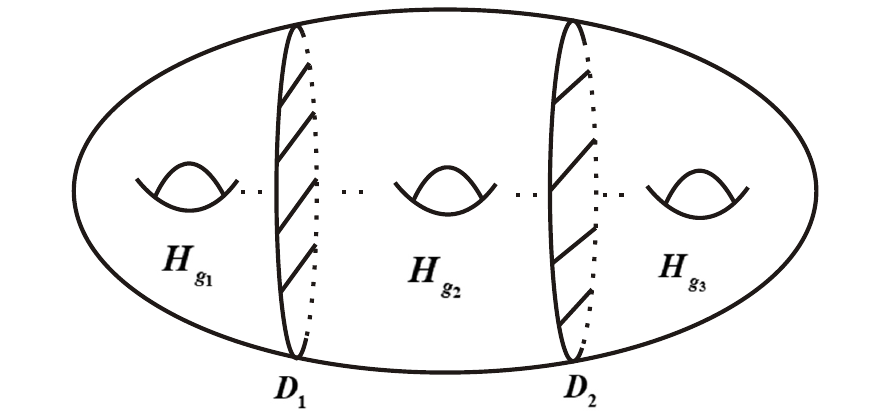}
\caption{Both of $D_1$ and $D_2$ are separating in $H_g$.}\label{fig4}
\end{figure}

As in Fig.\ref{fig4}, $\alpha_1\subset\partial H_{g_1}$ or $\alpha_1\subset\partial (H_{g_2}\cup H_{g_3})$, and $\alpha_2\subset\partial (H_{g_1}\cup H_{g_2})$ or $\alpha_2\subset\partial H_{g_3}$. So there are four subcases to be considered.

\textbf{Subcase 1.1:} $\alpha_1\subset \partial H_{g_1}$ and $\alpha_2\subset \partial H_{g_3}$.

Then $\alpha_1\cap \alpha_2=\phi$. At this time, $A_1\cap A_2=\phi$. So $A_1$ and $A_2$ are connected in $\mathcal{A}$.

\textbf{Subcase 1.2:} $\alpha_1\subset \partial H_{g_1}$ and $\alpha_2\subset \partial (H_{g_1}\cup H_{g_2})$.

If $A_1\cap A_2=\phi$, then $A_1$ and $A_2$ are connected in $\mathcal{A}$.  If $A_1\cap A_2\neq\phi$, then we can choose an essential arc $\alpha_3$ associated with $D_2$ in $\partial H_{g_3}$. Now the essential annulus $A_3=D_2\cup (\alpha_3\times [0,1])$ is disjoint from both $A_1$ and $A_2$. So $A_1-A_3-A_2$ is a path in $\mathcal{A}$ from $A_1$ to $A_2$.

\textbf{Subcase 1.3:} $\alpha_1\subset \partial (H_{g_2}\cup H_{g_3})$ and $\alpha_2\subset \partial H_{g_3}$.

If $A_1\cap A_2=\phi$, then $A_1$ and $A_2$ are connected in $\mathcal{A}$.  If $A_1\cap A_2\neq\phi$, then we can choose an essential arc $\alpha_3$ associated with $D_1$ in $\partial H_{g_1}$. Now the essential annulus $A_3=D_1\cup (\alpha_3\times [0,1])$ is disjoint from both $A_1$ and $A_2$. So $A_1-A_3-A_2$ is a path in $\mathcal{A}$ from $A_1$ to $A_2$.

\textbf{Subcase 1.4:} $\alpha_1\subset \partial (H_{g_2}\cup H_{g_3})$ and $\alpha_2\subset \partial (H_{g_1}\cup H_{g_2})$.

If $A_1\cap A_2=\phi$, then $A_1$ and $A_2$ are connected in $\mathcal{A}$.  If $A_1\cap A_2\neq\phi$, then we can choose an essential arc $\alpha_3$ associated with $D_1$ in $\partial H_{g_1}$, and choose an essential arc $\alpha_4$ associated with $D_2$ in $\partial H_{g_3}$. Now there are two essential annuli $A_3=D_1\cup (\alpha_3\times [0,1])$
and $A_4=D_2\cup (\alpha_4\times [0,1])$ in $H_g$. It is easy to see that $A_1\cap A_3=A_4\cap A_2=A_3\cap A_4=\phi$. So $A_1-A_3-A_4-A_2$ is a path in $\mathcal{A}$ from $A_1$ to $A_2$.

\textbf{Case 2:} One of $D_1$ and $D_2$ is separating and the other is nonseparating in $H_g$. Without loss of generality, assume $D_1$ is separating and $D_2$ is nonseparating in $H_g$.

Let $H_g\setminus D_1=H_{g_1}\cup H_{g_2}$ with $g_1+g_2=g$ and $g_i\geq 1 (i=1,2)$. Without loss of generality, assume $D_2\subset H_{g_2}$. See Fig.\ref{fig3}.

\begin{figure}[htbp]
\centering
\includegraphics[scale=0.6]{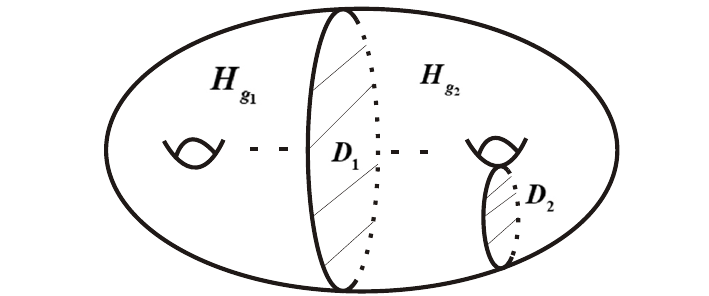}
\caption{$D_1$ is separating and $D_2$ is nonseparating in $H_g$.}\label{fig3}
\end{figure}

Now there are two subcases to be considered.

\textbf{Subcase 2.1:} $g(H_{g_2})=g_2\geq 2$.

Firstly, if $\alpha_2\cap D_1=\phi$, then $\alpha_2\subset \partial H_{g_2}$.
If $\alpha_1\cap \alpha_2=\phi$, then $A_1\cap A_2=\phi$ and $A_1$ and $A_2$ are connected in $\mathcal{A}$.
If $\alpha_1\cap \alpha_2\neq \phi$, then $\alpha_1\subset \partial H_{g_2}$. At this time, we can choose an essential arc, say $\alpha_3$, associated with $D_1$ in $\partial H_{g_1}$ to make the essential annulus $A_3=D_1\cup (\alpha_3\times [0,1])$ disjoint from both $A_1$ and $A_2$.
Thus, $A_1-A_3-A_2$ is a path in $\mathcal{A}$ from $A_1$ to $A_2$.

Secondly, if $\alpha_2\cap D_1\neq \phi$. Since $g(H_{g_2})=g_2\geq 2$, there exists an essential arc, say $\alpha_3$, associated with $D_2$ in $\partial H_{g_2}$ such that $\alpha_3\cap D_1=\phi$.
Let $A_3=D_2\cup (\alpha_3\times [0,1])$.
Then $A_3$ and $A_2$ are all essential annuli obtained by doing band-sum to $D_2$ along arcs. So by Theorem\ref{thm1}, there exists a path $A_3-\cdots-A_2$ in $\mathcal{A}$ from $A_3$ to $A_2$.
Now consider the arc $\alpha_1$.
If $\alpha_1\subset \partial H_{g_1}$, then $\alpha_1\cap \alpha_3=\phi$ and $A_1\cap A_3=\phi$. Thus, there exists a path $A_1-A_3-\cdots-A_2$ in $\mathcal{A}$ from $A_1$ to $A_2$.
If $\alpha_1\subset \partial H_{g_2}$, then there exists an essential arc, say $\alpha_4$, associated with $D_1$ in $\partial H_{g_1}$.
Let $A_4=D_1\cup (\alpha_4\times [0,1])$.
Since $D_1$ is separating in $H_g$, $\alpha_1$ and $\alpha_4$ are on different sides of $D_1$. So $A_1\cap A_4=\phi$ and $A_1-A_4$ is a path in $\mathcal{A}$ from $A_1$ to $A_4$. At this time, $A_4\cap A_3=\phi$.
Thus, there exists a path $A_1-A_4-A_3-\cdots-A_2$ in $\mathcal{A}$ from $A_1$ to $A_2$.

\textbf{Subcase 2.2:} $g(H_{g_2})=g_2=1$.

At this time, $H_{g_2}$ is a solid torus and $D_2$ is a meridian disk in $H_{g_2}$.
Since $g(H_{g_1})=g_1\geq 2$, there exists an essential separating disk, say $D_3$, in $H_{g_1}$ with $D_3\cap D_1=\phi$. Now $D_3$ is also an essential separating disk in $H_g$ which is disjoint from $D_2$.
As in Fig.\ref{fig12}, assume $H_g\setminus D_3=H^{1}\cup H^{2}$ and $D_2\in H^{2}$.
Since $g(H^{1})\geq 1$, we can choose an essential arc, say $\alpha_3$, associated with $D_3$ in $\partial H^{1}$
to make the annulus $A_3=D_3\cup (\alpha_3\times [0,1])$ essential in $H_g$.
For one thing, since $H_g\setminus D_3=H^{1}\cup H^{2}$, $g(H^{2})\geq 2$ and $D_2\in H^{2}$, similar to Subcase 2.1, there exists a path $A_3-\cdots-A_2$ in $\mathcal{A}$ from $A_3$ to $A_2$.
For another, since $D_1$ and $D_3$ are disjoint separating essential disks in $H_g$, similar to Case 1, there exists a path $A_1-\cdots-A_3$ in $\mathcal{A}$ from $A_1$ to $A_3$.
Thus, there exists a path $A_1-\cdots-A_3-\cdots-A_2$ in $\mathcal{A}$ from $A_1$ to $A_2$.

\begin{figure}[htbp]
\centering
\includegraphics[scale=0.6]{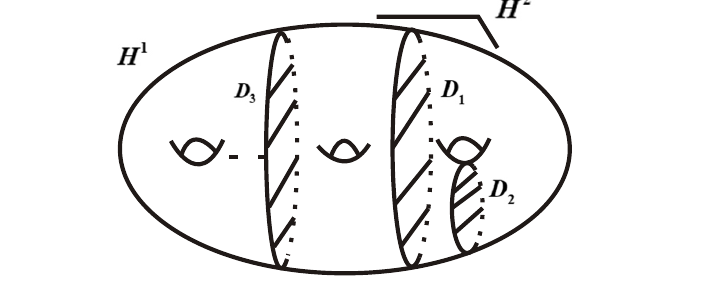}
\caption{There exists an essential separating disk $D_3$.}\label{fig12}
\end{figure}

\textbf{Case 3:} Both $D_1$ and $D_2$ are nonseparating in $H_g$.

Now there are two subcases to be considered.

\textbf{Subcase 3.1:} $H_g\setminus (D_1\cup D_2)$ has two components. Let $H_g\setminus (D_1\cup D_2)=H_{g_1}\cup H_{g_2}$ with $g_i\geq 1 (i=1,2)$ and $g_1+g_2=g-1\geq 2$. See Fig.\ref{fig5}.

\begin{figure}[htbp]
\centering
\includegraphics[scale=0.6]{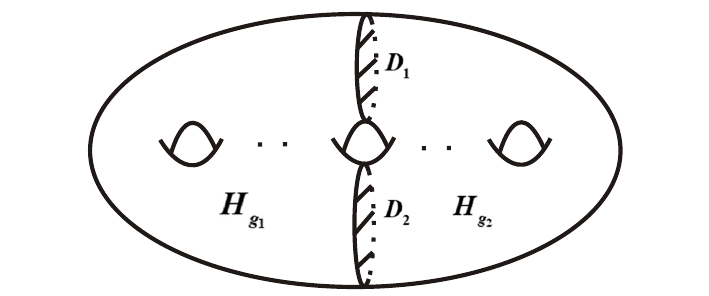}
\caption{The union of $D_1$ and $D_2$ is separating in $H_g$.}\label{fig5}
\end{figure}

When $A_1\cap A_2=\phi$, $A_1$ and $A_2$ are connected in $\mathcal{A}$.

When $A_1\cap A_2\neq \phi$. As in Fig.\ref{fig5}, since $g_i\geq 1 (i=1,2)$, we can choose an essential arc, say $\alpha_3$, associated with $D_1$ in $\partial H_{g_1}$ such that $\alpha_3\cap D_2=\phi$, and choose an
essential arc, say $\alpha_4$, associated with $D_2$ in $\partial H_{g_2}$ such that $\alpha_4\cap D_1=\phi$.
So $\alpha_3\cap \alpha_4=\phi$.
Let $A_3=D_1\cup (\alpha_3\times [0,1])$ and
$A_4=D_2\cup (\alpha_4\times [0,1])$. It is easy to see that $A_3\cap A_4=\phi$.
Since $A_1$ and $A_3$ are all obtained by doing band-sum to $D_1$, by Theorem\ref{thm1}, there exists a path $A_1-\cdots-A_3$ in $\mathcal{A}$ from $A_1$ to $A_3$.
Similarly, since $A_4$ and $A_2$ are all obtained by doing band-sum to $D_2$, by Theorem\ref{thm1}, there exists a path $A_4-\cdots-A_2$ in $\mathcal{A}$ from $A_4$ to $A_2$.
So there exists a path $A_1-\cdots-A_3-A_4-\cdots-A_2$ in $\mathcal{A}$ from $A_1$ to $A_2$.

\textbf{Subcase 3.2:} $H_g\setminus (D_1\cup D_2)$ is a handlebody. Let $H_g\setminus (D_1\cup D_2)=H'$. So $g(H')=g-2\geq 1$. See Fig.\ref{fig6}.

\begin{figure}[htbp]
\centering
\includegraphics[scale=0.6]{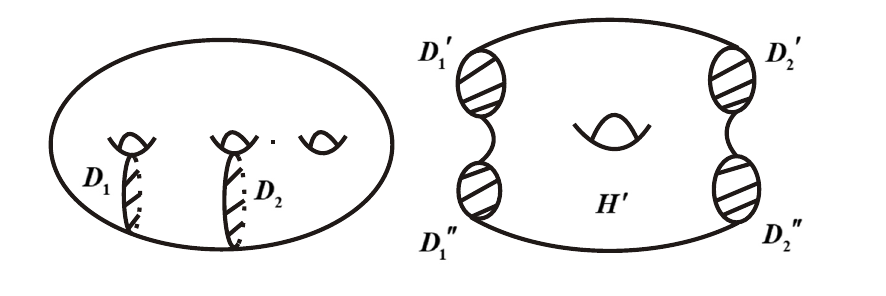}
\caption{The union of $D_1$ and $D_2$ is nonseparating in $H_g$.}\label{fig6}
\end{figure}

As in Fig.\ref{fig6}, for $i=1,2$, assume the two cutting sections of $D_i$ on $\partial H'$ are denoted by $D_i ^{'}$ and $D_i ^{''}$.
Without loss of generality, assume $\partial \alpha_1\in \partial D_1 ^{'}$ and
$\partial \alpha_2\in \partial D_2 ^{'}$.

If $A_1\cap A_2=\phi$, then $A_1$ and $A_2$ are connected in $\mathcal{A}$. If $A_1\cap A_2\neq\phi$, then we get a path from $A_1$ to $A_2$ as follows.
Let $S_{g-2,4}=\partial H'\setminus (D_1 ^{'}\cup D_2 ^{'}\cup D_1 ^{''}\cup D_2 ^{''})$. So $S_{g-2,4}$ is an orientable surface with genus $g-2$ and 4 boundary components.
Since $g\geq 3$, $\chi(S_{g-2,4})=2-2(g-2)-4=2-2g\leq -4$. So there exist two disjoint essential arcs, say $\alpha_3$ and $\alpha_4$, in $S_{g-2,4}$ such that $\partial \alpha_3 \in \partial D_1 ^{'}$, $\partial \alpha_4 \in \partial D_2 ^{'}$.
$A_3=D_1\cup(\alpha_3\times [0,1])$ and
$A_4=D_2\cup(\alpha_4\times [0,1])$ are all essential annuli in $H_g$.
Since $\alpha_3\cap \alpha_4=\phi$, $A_3\cap A_4=\phi$ and $A_3$ and $A_4$ are connected in $\mathcal{A}$.
Since $A_1$ and $A_3$ are all obtained by doing band-sum to $D_1$, by Theorem\ref{thm1}, there exists a path $A_1-\cdots-A_3$ in $\mathcal{A}$ from $A_1$ to $A_3$.
Similarly, since $A_4$ and $A_2$ are all obtained by doing band-sum to $D_2$, by Theorem\ref{thm1}, there exists a path $A_4-\cdots-A_2$ in $\mathcal{A}$ from $A_4$ to $A_2$.
So there exists a path $A_1-\cdots-A_3-A_4-\cdots-A_2$ in $\mathcal{A}$ from $A_1$ to $A_2$.

\end{proof}

The following is our main theorem.

\begin{thm}\label{thm3}
Let $H_g$ be an orientable handlebody of genus $g\geq 3$. Then the annulus complex $\mathcal{A}$ of $H_g$ is connected.
\end{thm}

\begin{proof}

Choose any two essential annuli $A_1$ and $A_2$ in $\mathcal{A}$. Let $A_1=D_1\cup (\alpha_1\times [0,1])$ and $A_2=D_2\cup (\alpha_2\times [0,1])$.

There are two cases to be considered.

\textbf{Case 1:} $D_1\cap D_2=\phi$.

If $D_1=D_2=D$, then by Theorem\ref{thm1}, $A_1$ and $A_2$ are connected by a path in $\mathcal{A}$.

If $D_1$ and $D_2$ are different essential disks in $H_g$, then by Theorem\ref{thm2}, $A_1$ and $A_2$ are connected by a path in $\mathcal{A}$.

\textbf{Case 2:} $D_1\cap D_2\neq \phi$.

If $D_1\cap D_2\neq \phi$, then by the connectivity of the disk complex $\mathcal{D}$, there exists a path $D_1=E_1-\cdots-E_j-\cdots-E_n=D_2$ in $\mathcal{D}$ from $D_1$ to $D_2$.
For each $2\leq j\leq n-1$, since $E_j$ is an essential disk in $H_g$, there exists an essential arc, say $\beta_j$, associated with $E_j$ in $\partial H_g$.
Now for each $2\leq j\leq n-1$,
$A_j^{'}=E_j\cup (\beta_j\times [0,1])$
is an essential annulus in $H_g$.
Let $\beta_1=\alpha_1$, $\beta_n=\alpha_2$, $A_1^{'}=A_1$ and $A_n^{'}=A_2$.
Thus, for $j=1,\cdots,n-1$, the essential annuli $A_j ^{'}$ and $A_{j+1} ^{'}$ are obtained by doing band-sum to disjoint essential disks $E_j$ and $E_{j+1}$ along arcs $\beta_i$ and $\beta_{i+1}$, respectively. By Theorem\ref{thm2}, $A_j ^{'}$ and $A_{j+1} ^{'}$ are connected by a path $A_j ^{'}-\cdots-A_{j+1} ^{'}$ in $\mathcal{A}$.
Thus, there exists a path $A_1=A_1 ^{'}-\cdots-A_2 ^{'}-\cdots-A_j ^{'}-\cdots-A_{n-1} ^{'}-\cdots-A_{n} ^{'}=A_2$ in $\mathcal{A}$ from $A_1$ to $A_2$. So $A_1$ and $A_2$ are connected in $\mathcal{A}$.

\end{proof}

\end{document}